\documentclass[12pt,reqno]{article}

\usepackage[usenames]{color}
\usepackage[colorlinks=true,
linkcolor=webgreen, filecolor=webbrown,
citecolor=webgreen]{hyperref}

\definecolor{webgreen}{rgb}{0,.5,0}
\definecolor{webbrown}{rgb}{.6,0,0}

\usepackage{amssymb}
\usepackage{graphicx}
\usepackage{amscd}
\usepackage{lscape}
\usepackage{tikz}
\usepackage{tikz-cd}
\usepackage{pgfplots}
\usepackage{mathtools}

\usetikzlibrary{matrix}
\usetikzlibrary{fit,shapes}
\usetikzlibrary{positioning, calc}
\tikzset{circle node/.style = {circle,inner sep=1pt,draw, fill=white},
        X node/.style = {fill=white, inner sep=1pt},
        dot node/.style = {circle, draw, inner sep=5pt}
        }
\usepackage{tkz-fct}

\usepackage{amsthm}
\newtheorem{theorem}{Theorem}

\newtheorem{proposition}[theorem]{Proposition}
\newtheorem{corollary}[theorem]{Corollary}

\theoremstyle{definition}

\usepackage{float}

\usepackage{graphics,amsmath}
\usepackage{amsfonts}
\usepackage{latexsym}
\usepackage{epsf}

\newcommand{\seqnum}[1]{\href{http://oeis.org/#1}{\underline{#1}}}

\begin{document}

\begin{center}
\vskip 1cm{\LARGE\bf Generalized Eulerian Triangles and Some Special Production Matrices} \vskip 1cm \large
Paul Barry\\
School of Science\\
Waterford Institute of Technology\\
Ireland\\
\href{mailto:pbarry@wit.ie}{\tt pbarry@wit.ie}
\end{center}
\vskip .2 in

\begin{abstract} We show how some special production matrices may be used to define families of generalized Eulerian triangles. We furthermore show that these generalized Eulerian triangles are the coefficient arrays of polynomials which are the moments of families of orthogonal polynomials. Using the previously defined $\mathcal{T}$ transform, we associate these generalized Eulerian triangles to triangles defined by Catalan generating functions. Again, these new triangles are the coefficient arrays of polynomial moment sequences. The main tools used are the Sumudu transform, Jacobi continued fractions and Riordan arrays.
\end{abstract}

\section{Preliminaries}
In this note, we shall generalize the Eulerian triangles of type $A$ and $B$ \cite{Petersen} by using parameters arising from the production matrices \cite{ProdMat_0, ProdMat} of special number triangles that are related to the original Eulerian triangles through the Sumudu transform \cite{Belgacem2, Belgacem1, Watugala} and the taking of reciprocals. Regarding each generalized Eulerian triangle as the coefficient matrix of a family of polynomials allows us to show that the resulting polynomials are the moments for a related family of orthogonal polynomials \cite{Chihara, Gautschi, Szego}. We use techniques of exponential Riordan arrays \cite{Book, DeutschShap} to show this. We also relate, via the $\mathcal{T}$ transform \cite{T}, these newly found polynomial families to families of polynomials defined by ordinary Riordan arrays \cite{Book, Survey, SGWW, Spru}. We have previously shown that the standard Eulerian polynomials are moments \cite{Eulerian} using techniques of Riordan arrays and orthogonal polynomials \cite{Classical, Barry_Meixner}. We find the language of continued fractions useful in this context \cite{CFT, Wall}.
The notation
$$\mathcal{J}(a,b,c,\ldots; \alpha, \beta, \gamma,\ldots)$$ signifies a Jacobi-type continued fraction
$$\cfrac{1}{1- ax-
\cfrac{\alpha x^2}{1-bx-
\cfrac{\beta x^2}{1- cx-
\cfrac{\gamma x^2}{1-\cdots}}}}.$$
The generating function of the Catalan numbers $C_n=\frac{1}{n+1}\binom{2n}{n}$ will be denoted by $c(x)$. We have $$c(x)=\frac{1-\sqrt{1-4x}}{2x}.$$
 
Integer sequences will be referred to by their $Annnnnn$ in the On-Line Encyclopedia of Integer Sequences \cite{SL1, SL2}. 

\section{Introduction}
The three Eulerian triangles of type $A$, which we shall denote by $E_1$, $E_2$ and $E_3$, have bivariate generating functions, respectively, of
$$\frac{1-r}{e^{rt}-re^t},\quad \frac{(1-r)e^{rt}}{e^{rt}-re^t},\quad\text{and} \quad\frac{(r-1)^2 e^{(r+1)t}}{(e^{rt}-re^t)^2}.$$ For instance, $E_3$ is the centrally symmetric (or Pascal-like) triangle that begins
$$\left(
\begin{array}{ccccccc}
 1 & 0 & 0 & 0 & 0 & 0 & 0 \\
 1 & 1 & 0 & 0 & 0 & 0 & 0 \\
 1 & 4 & 1 & 0 & 0 & 0 & 0 \\
 1 & 11 & 11 & 1 & 0 & 0 & 0 \\
 1 & 26 & 66 & 26 & 1 & 0 & 0 \\
 1 & 57 & 302 & 302 & 57 & 1 & 0 \\
 1 & 120 & 1191 & 2416 & 1191 & 120 & 1 \\
\end{array}
\right).$$ These are, respectively, \seqnum{A173018}, \seqnum{A123125}, and \seqnum{A008292}. 
These generating functions are exponential in $t$ and ordinary in $r$. For each of these generating functions,
we calculate the Sumudu transform (in $t$) of its reciprocal (multiplicative inverse). We obtain, respectively, the following ordinary generating functions
$$\frac{1-(r+1)x}{(1-x)(1-rx)}, \quad \frac{1-x}{1+(r-1)x},\quad\text{and}\quad \frac{1-(r+1)x+2rx^2}{1-(r-1)^2x^2}.$$
These are the generating functions of the three number triangles that begin, respectively, as
$$\left(
\begin{array}{ccccccc}
 1 & 0 & 0 & 0 & 0 & 0 & 0 \\
 0 & 0 & 0 & 0 & 0 & 0 & 0 \\
 0 & -1 & 0 & 0 & 0 & 0 & 0 \\
 0 & -1 & -1 & 0 & 0 & 0 & 0 \\
 0 & -1 & -1 & -1 & 0 & 0 & 0 \\
 0 & -1 & -1 & -1 & -1 & 0 & 0 \\
 0 & -1 & -1 & -1 & -1 & -1 & 0 \\
\end{array}
\right),$$
$$\left(
\begin{array}{ccccccc}
 1 & 0 & 0 & 0 & 0 & 0 & 0 \\
 0 & -1 & 0 & 0 & 0 & 0 & 0 \\
 0 & -1 & 1 & 0 & 0 & 0 & 0 \\
 0 & -1 & 2 & -1 & 0 & 0 & 0 \\
 0 & -1 & 3 & -3 & 1 & 0 & 0 \\
 0 & -1 & 4 & -6 & 4 & -1 & 0 \\
 0 & -1 & 5 & -10 & 10 & -5 & 1 \\
\end{array}
\right),$$
and
$$\left(
\begin{array}{ccccccc}
 1 & 0 & 0 & 0 & 0 & 0 & 0 \\
 -1 & -1 & 0 & 0 & 0 & 0 & 0 \\
 1 & 0 & 1 & 0 & 0 & 0 & 0 \\
 -1 & 1 & 1 & -1 & 0 & 0 & 0 \\
 1 & -2 & 2 & -2 & 1 & 0 & 0 \\
 -1 & 3 & -2 & -2 & 3 & -1 & 0 \\
 1 & -4 & 7 & -8 & 7 & -4 & 1 \\
\end{array}
\right).$$
We call these triangles $T_1$, $T_2$ and $T_3$, respectively. With regard to $T_3$, we note that
$$ \frac{1-(r+1)x+2rx^2}{1-(r-1)^2x^2}=\frac{1}{1-(r-1)x}-\frac{2rx(1-x)}{1-(r-1)^2x^2}.$$
Thus the general term of $T_3$ is given by
$$(-1)^{n-k}\binom{n}{k}-2 (-1)^{n-k}\binom{2 \lfloor \frac{n-1}{2} \rfloor}{ 2 \lfloor \frac{n-1}{2} \rfloor-k+1}.$$

The Eulerian numbers of type $B$ \seqnum{A060187} have generating function
$$\frac{(1-r) e^{(1-r) t}}{1-r e^{2 (1-r) t}}.$$ They define the number triangle $E_B$ that begins
$$\left(
\begin{array}{ccccccc}
 1 & 0 & 0 & 0 & 0 & 0 & 0 \\
 1 & 1 & 0 & 0 & 0 & 0 & 0 \\
 1 & 6 & 1 & 0 & 0 & 0 & 0 \\
 1 & 23 & 23 & 1 & 0 & 0 & 0 \\
 1 & 76 & 230 & 76 & 1 & 0 & 0 \\
 1 & 237 & 1682 & 1682 & 237 & 1 & 0 \\
 1 & 722 & 10543 & 23548 & 10543 & 722 & 1 \\
\end{array}
\right).$$
We again take the Sumudu transform of the reciprocal of the above generating function to get
$$\frac{1-(r+1)x}{1-(r-1)^2x^2}.$$ This is the generating function of the triangle $T_B$ that begins
$$\left(
\begin{array}{ccccccc}
 1 & 0 & 0 & 0 & 0 & 0 & 0 \\
 -1 & -1 & 0 & 0 & 0 & 0 & 0 \\
 1 & -2 & 1 & 0 & 0 & 0 & 0 \\
 -1 & 1 & 1 & -1 & 0 & 0 & 0 \\
 1 & -4 & 6 & -4 & 1 & 0 & 0 \\
 -1 & 3 & -2 & -2 & 3 & -1 & 0 \\
 1 & -6 & 15 & -20 & 15 & -6 & 1 \\
\end{array}
\right).$$ This is equal to
$$\left(
\begin{array}{ccccccc}
 1 & 0 & 0 & 0 & 0 & 0 & 0 \\
 -1 & 1 & 0 & 0 & 0 & 0 & 0 \\
 1 & -2 & 1 & 0 & 0 & 0 & 0 \\
 -1 & 3 & -3 & 1 & 0 & 0 & 0 \\
 1 & -4 & 6 & -4 & 1 & 0 & 0 \\
 -1 & 5 & -10 & 10 & -5 & 1 & 0 \\
 1 & -6 & 15 & -20 & 15 & -6 & 1 \\
\end{array}
\right)-\left(
\begin{array}{ccccccc}
 0 & 0 & 0 & 0 & 0 & 0 & 0 \\
 0 & -1 & 0 & 0 & 0 & 0 & 0 \\
 0 & 0 & 0 & 0 & 0 & 0 & 0 \\
 0 & -1 & 2 & -1 & 0 & 0 & 0 \\
 0 & 0 & 0 & 0 & 0 & 0 & 0 \\
 0 & -1 & 4 & -6 & 4 & -1 & 0 \\
 0 & 0 & 0 & 0 & 0 & 0 & 0 \\
\end{array}
\right).$$
We calculate the production matrices of the triangles $T_3^{-1}$ and $T_B^{-1}$. We get, respectively, production matrices that begin
$$\left(
\begin{array}{ccccccc}
 -1 & -1 & 0 & 0 & 0 & 0 & 0 \\
 2 & 1 & -1 & 0 & 0 & 0 & 0 \\
 -4 & -2 & 1 & -1 & 0 & 0 & 0 \\
 4 & 4 & 2 & 1 & -1 & 0 & 0 \\
 -8 & -8 & -4 & -2 & 1 & -1 & 0 \\
 8 & 8 & 4 & 4 & 2 & 1 & -1 \\
 -16 & -16 & -8 & -8 & -4 & -2 & 1 \\
\end{array}
\right)$$
and
$$\left(
\begin{array}{ccccccc}
 -1 & -1 & 0 & 0 & 0 & 0 & 0 \\
 4 & 3 & -1 & 0 & 0 & 0 & 0 \\
 0 & 0 & -1 & -1 & 0 & 0 & 0 \\
 0 & 0 & 4 & 3 & -1 & 0 & 0 \\
 0 & 0 & 0 & 0 & -1 & -1 & 0 \\
 0 & 0 & 0 & 0 & 4 & 3 & -1 \\
 0 & 0 & 0 & 0 & 0 & 0 & -1 \\
\end{array}
\right).$$
In each case, the structure is that of an initial column, followed by alternating columns that are the same (subject to shifting). They can be seen as generalizing the production matrices of ordinary Riordan arrays. We shall use this structure in the next two sections to define generalized Eulerian triangles of type $A$ and $B$.
\section{Generalized Eulerian triangles of type $A$}
To generalize $E_3$, we generalize the production matrix
$$\left(
\begin{array}{ccccccc}
 -1 & -1 & 0 & 0 & 0 & 0 & 0 \\
 2 & 1 & -1 & 0 & 0 & 0 & 0 \\
 -4 & -2 & 1 & -1 & 0 & 0 & 0 \\
 4 & 4 & 2 & 1 & -1 & 0 & 0 \\
 -8 & -8 & -4 & -2 & 1 & -1 & 0 \\
 8 & 8 & 4 & 4 & 2 & 1 & -1 \\
 -16 & -16 & -8 & -8 & -4 & -2 & 1 \\
\end{array}
\right)$$ to get the parameterized triangle that begins
$$\left(
\begin{array}{ccccccc}
 -1 & -1 & 0 & 0 & 0 & 0 & 0 \\
 a & 1 & -1 & 0 & 0 & 0 & 0 \\
 -a^2 & -a & 1 & -1 & 0 & 0 & 0 \\
 a^2 & a^2 & a & 1 & -1 & 0 & 0 \\
 -a^3 & -a^3 & -a^2 & -a & 1 & -1 & 0 \\
 a^3 & a^3 & a^2 & a^2 & a & 1 & -1 \\
 -a^4 & -a^4 & -a^3 & -a^3 & -a^2 & -a & 1 \\
\end{array}
\right).$$
We find that only for $a=0$ and $a=2$ is the resulting triangle (the inverse of the triangle with the given production matrix) centrally symmetric. The case $a=2$ is the case of $E_3$.
\begin{proposition} For $a=0$, the resulting triangle (the inverse of the triangle with the given production matrix) is the Riordan array
$$\left(\frac{1-2x}{1-x}, \frac{x}{x-1}\right).$$
\end{proposition}
\begin{proof} This follows immediately since the production matrix takes the form
$$\left(
\begin{array}{ccccccc}
 -1 & -1 & 0 & 0 & 0 & 0 & 0 \\
 0 & 1 & -1 & 0 & 0 & 0 & 0 \\
 0 & 0 & 1 & -1 & 0 & 0 & 0 \\
 0 & 0 & 0 & 1 & -1 & 0 & 0 \\
 0 & 0 & 0 & 0 & 1 & -1 & 0 \\
 0 & 0 & 0 & 0 & 0 & 1 & -1 \\
 0 & 0 & 0 & 0 & 0 & 0 & 1 \\
\end{array}
\right).$$
\end{proof}
\begin{corollary} The triangle defined by $a=0$ has bivariate generating function
$$\frac{1-2x}{1+(y-1)x}.$$
\end{corollary}
\begin{proof} The generating function of the Riordan array $\left(\frac{1-2x}{1-x}, \frac{x}{x-1}\right)$ is given by
$$\frac{\frac{1-2x}{1-x}}{1-y \frac{x}{x-1}}.$$
\end{proof}
\begin{proposition} The reciprocal of the inverse Sumudu transform of
$$\frac{1-2x}{1+(y-1)x}$$ is given by
$$\frac{e^{t(y-1)}(y-1)}{y+1-2 e^{t(y-1)}}.$$
\end{proposition}
\begin{proof} Direct calculation.
\end{proof}
We can now define the generalized Eulerian triangle corresponding to $a=0$ to be the triangle with generating function $$\frac{e^{t(y-1)}(y-1)}{y+1-2 e^{t(y-1)}}.$$
This triangle begins
$$\left(
\begin{array}{ccccccc}
 1 & 0 & 0 & 0 & 0 & 0 & 0 \\
 1 & 1 & 0 & 0 & 0 & 0 & 0 \\
 3 & 4 & 1 & 0 & 0 & 0 & 0 \\
 13 & 23 & 11 & 1 & 0 & 0 & 0 \\
 75 & 166 & 116 & 26 & 1 & 0 & 0 \\
 541 & 1437 & 1322 & 482 & 57 & 1 & 0 \\
 4683 & 14512 & 16563 & 8408 & 1793 & 120 & 1 \\
\end{array}
\right).$$ 
The row sums of this triangle are $2^n n!$, \seqnum{A000165}. 
\begin{proposition} The generalized Eulerian polynomials with coefficient array defined by $\frac{e^{t(y-1)}(y-1)}{y+1-2 e^{t(y-1)}}$ are the moments for the family of orthogonal polynomials whose coefficient array is given by
$$\left[\frac{e^{t(y-1)}(y-1)}{y+1-2 e^{t(y-1)}},\frac{e^t-e^{t y}}{2 e^{r y}-e^t (y+1)}\right]^{-1}=\left[\frac{1}{1+t(y+1)},\frac{1}{1-y}\ln\left(\frac{1+2t}{1+t(y+1)}\right)\right].$$
\end{proposition}
\begin{proof} We must show that the production matrix of the exponential Riordan array
$$\left[\frac{e^{t(y-1)}(y-1)}{y+1-2 e^{t(y-1)}},\frac{e^t-e^{t y}}{2 e^{r y}-e^t (y+1)}\right]$$  is tri-diagonal. We find that
$$A(t)=(1+2t)(1+(y+1)t), \quad Z(t)=(y+1)(1+2t).$$ Thus the production matrix is tri-diagonal, as required.
\end{proof}
\begin{corollary} The moments have an ordinary generating function given by the continued fraction
$$\mathcal{J}(y+1,2y+4,3y+7,\ldots; 2(y+1), 8(y+1), 18(y+1),\ldots).$$
\end{corollary}
\begin{proof} This follows from the form of $A(t)$ and $Z(t)$.
\end{proof}
In the Del\'eham notation, the above triangle is given by
$$[1, 2, 2, 4, 3, 6, 4, 8, 5,\ldots]\,\Delta\,[1,0,2,0,3,0,4,0,5,\ldots].$$
The first column numbers are the Fubini numbers (ordered partitions) \seqnum{A000670}.
We observe now that the $\mathcal{T}$ transform of the moment triangle has generating function
$$\mathcal{J}(y+1,y+3,y+3,\ldots; 2(y+1), 2(y+1), 2(y+1),\ldots).$$
The corresponding triangle \seqnum{A114608} (the coefficient array of the moment polynomials) begins
$$\left(
\begin{array}{ccccccc}
 1 & 0 & 0 & 0 & 0 & 0 & 0 \\
 1 & 1 & 0 & 0 & 0 & 0 & 0 \\
 3 & 4 & 1 & 0 & 0 & 0 & 0 \\
 11 & 19 & 9 & 1 & 0 & 0 & 0 \\
 45 & 96 & 66 & 16 & 1 & 0 & 0 \\
 197 & 501 & 450 & 170 & 25 & 1 & 0 \\
 903 & 2668 & 2955 & 1520 & 365 & 36 & 1 \\
\end{array}
\right).$$ 

This counts the number of bi-colored Dyck paths of semi-length $n$ and having $k$ peaks of the form $UD$ (Deutsch, \cite{SL1}). In the Del\'eham notation, this is
$$[1,2,1,2,1,2,1,\ldots]\, \Delta \,[1,0,1,0,1,0,1,0,1,\ldots].$$

Standard Riordan array techniques now give us the following result. 
\begin{proposition} The polynomials with coefficient array with generating function 
$$\mathcal{J}(y+1,y+3,y+3,\ldots; 2(y+1), 2(y+1), 2(y+1),\ldots)$$ are the moments of the family of orthogonal polynomials whose coefficient array is given by the Riordan array 
$$\left(\frac{1+2x}{1+(y+3)x+2(y+1)x^2}, \frac{x}{1+(y+3)x+2(y+1)x^2}\right).$$ 
The inverse of this matrix, or 
$$\left(\frac{1}{1-(y-1)x}c\left(\frac{2x}{(1-(y-1)x)^2}\right), \frac{x}{1-(y+3)x}c\left(\frac{2x^2(y+1)}{(1-(y+3)x)^2}\right)\right),$$ is the moment matrix for these polynomials.
The moments have generating function $$\frac{1}{1-(y-1)x}c\left(\frac{2x}{(1-(y-1)x)^2}\right).$$
\end{proposition}

\section{Generalized Eulerian triangles of type $B$}
In this section we start with the production matrix that begins 
$$\left(
\begin{array}{cccccccc}
 -1 & -1 & 0 & 0 & 0 & 0 & 0 & 0 \\
 b & a & -1 & 0 & 0 & 0 & 0 & 0 \\
 0 & 0 & -1 & -1 & 0 & 0 & 0 & 0 \\
 0 & 0 & b & a & -1 & 0 & 0 & 0 \\
 0 & 0 & 0 & 0 & -1 & -1 & 0 & 0 \\
 0 & 0 & 0 & 0 & b & a & -1 & 0 \\
 0 & 0 & 0 & 0 & 0 & 0 & -1 & -1 \\
 0 & 0 & 0 & 0 & 0 & 0 & b & a \\
\end{array}
\right).$$ This produces a matrix whose inverse begins 
$$\left(
\begin{array}{ccccc}
 1 & 0 & 0 & 0 & 0 \\
 -1 & -1 & 0 & 0 & 0 \\
 b-a & 1-a & 1 & 0 & 0 \\
 a-b & 2 a-b-1 & a-2 & -1 & 0 \\
 (a-b)^2 & 2 (a-1) (a-b) & a^2-4 a+2 b+1 & 2 (1-a) & 1 \\
\end{array}
\right).$$ This triangle has generating function 
$$\frac{1-(r+1)x}{1-((r+1)(r-a)+b)x^2}.$$ 
For $a=3$ and $b=4$, we recover 
$$\frac{(1-(r+1)x}{1-(r-1)^2 x^2}.$$ Taking the reciprocal of the inverse Sumudu transform of this generating function, we obtain the follow bivariate exponential generating function.
$$G(t,r;a,b)=\frac{2 e^{\sqrt{b+(r+1)(r-a)}t}\sqrt{b+(r+1)(r-a)}}{r+1+\sqrt{b+(r+1)(r-a)}-2e^{2 \sqrt{b+(r+1)(r-a)}t}(r+1-\sqrt{b+(r+1)(r-a)})}.$$ 
This expands to give the triangle that begins 
\begin{scriptsize}
$$\left(
\begin{array}{ccccc}
 1 & 0 & 0 & 0 & 0 \\
 1 & 1 & 0 & 0 & 0 \\
 a-b+2 & a+3 & 1 & 0 & 0 \\
 5 a-5 b+6 & 10 a-5 b+13 & 5 a+8 & 1 & 0 \\
 5 a^2-10 a b+28 a+5 b^2-28 b+24 & 10 a^2-10 a b+74 a-46 b+68
   & 5 a^2+64 a-18 b+65 & 18 a+22 & 1 \\
\end{array}
\right).$$ 
\end{scriptsize}
We call this the generalized Eulerian triangle of type B defined by the $(a,b)$-parameterized production matrix. 
\begin{proposition} The  generalized Eulerian polynomials with coefficient array defined by $G(t,r;a,b)$ are the moments for the family of orthogonal polynomials whose coefficient array is given by the exponential Riordan array 
$$\left[\frac{1}{\sqrt{1+2(r+1)x+((r+1)(a+1)-b)x^2}}, 
\frac{\tan^{-1}\left(\frac{r+1+x((r+1)(a+1)-b)}{\sqrt{(r+1)(a-r)-b}}\right)-\sin^{-1}\left(\frac{r+1}{\sqrt{(r+1)(a+1)-b}}\right)}{\sqrt{(r+1)(a-r)-b}}
\right].$$ 
The inverse of this array is the corresponding moment matrix, which is given by 
$$\left[G(x,r;a,b), \frac{\sqrt{(r+1)(a-r)-b}\tan\left(\sin^{-1}\left(\frac{r+1}{\sqrt{(r+1)(a+1)-b}}\right)+x \sqrt{(r+1)(a-r)-b}\right)-r-1}{(r+1)(a+1)-b}\right].$$
\end{proposition}
\begin{proof} We must show that the production matrix of the moment matrix is tri-diagonal. This is so since we find that 
$$Z(x)=r+1+((r+1)(a+1)-b)x,$$ 
and
$$A(x)=1+2(r+1)x+((r+1)(a+1)-b)x^2.$$ 
\end{proof}
\begin{corollary} The ordinary generating function of the generalized Eulerian triangles for $(a,b)$ is given by the continued fraction 
$$\mathcal{J}(r+1,3(r+1),5(r+1),\ldots; (r+1)(a+1)-b, 4((r+1)(a+1)-b),9((r+1)(a+1)-b),\ldots).$$
\end{corollary}

For $a=0$ and $b=1$, we get the triangle that begins 
$$\left(
\begin{array}{ccccccc}
 1 & 0 & 0 & 0 & 0 & 0 & 0 \\
 1 & 1 & 0 & 0 & 0 & 0 & 0 \\
 1 & 3 & 1 & 0 & 0 & 0 & 0 \\
 1 & 8 & 8 & 1 & 0 & 0 & 0 \\
 1 & 22 & 47 & 22 & 1 & 0 & 0 \\
 1 & 63 & 245 & 245 & 63 & 1 & 0 \\
 1 & 185 & 1210 & 2113 & 1210 & 185 & 1 \\
\end{array}
\right).$$
For $a=1$, $b=1$ we get 
$$\left(
\begin{array}{ccccccc}
 1 & 0 & 0 & 0 & 0 & 0 & 0 \\
 1 & 1 & 0 & 0 & 0 & 0 & 0 \\
 2 & 4 & 1 & 0 & 0 & 0 & 0 \\
 6 & 18 & 13 & 1 & 0 & 0 & 0 \\
 24 & 96 & 116 & 40 & 1 & 0 & 0 \\
 120 & 600 & 1020 & 660 & 121 & 1 & 0 \\
 720 & 4320 & 9480 & 9120 & 3542 & 364 & 1 \\
\end{array}
\right).$$
For $a=1$, $b=2$ we get the triangle that begins
$$\left(
\begin{array}{ccccccc}
 1 & 0 & 0 & 0 & 0 & 0 & 0 \\
 1 & 1 & 0 & 0 & 0 & 0 & 0 \\
 1 & 4 & 1 & 0 & 0 & 0 & 0 \\
 1 & 13 & 13 & 1 & 0 & 0 & 0 \\
 1 & 40 & 98 & 40 & 1 & 0 & 0 \\
 1 & 121 & 602 & 602 & 121 & 1 & 0 \\
 1 & 364 & 3363 & 6488 & 3363 & 364 & 1 \\
\end{array}
\right).$$ 
By the $\mathcal{T}$ transform, we can associate to these triangles the triangles with generating function 
$$\mathcal{J}(r+1,2(r+1),2(r+1),\ldots; (r+1)(a+1)-b, ((r+1)(a+1)-b), ((r+1)(a+1)-b),\ldots).$$
\begin{proposition} We have 
$$\mathcal{T}(G(x,r;a,b))=c(x(r+1+((r+1)(a-r)-b)x).$$ 
\end{proposition}
\begin{proof} This follows by considering the continued fraction above. Thus the sought generating function $g(x)$ satisfies 
$$g(x)=\frac{1}{1+(r+1)x+((r+1)(a+1)-b)x^2 u(x)},$$ where 
$$u(x)=\frac{1}{1+2(r+1)x+((r+1)(a+1)-b)x^2 u(x)}.$$ 
\end{proof}

For $a=0$, $b=0$, we get the triangle that begins 
$$\left(
\begin{array}{ccccccc}
 1 & 0 & 0 & 0 & 0 & 0 & 0 \\
 1 & 1 & 0 & 0 & 0 & 0 & 0 \\
 2 & 3 & 1 & 0 & 0 & 0 & 0 \\
 5 & 11 & 7 & 1 & 0 & 0 & 0 \\
 14 & 41 & 41 & 15 & 1 & 0 & 0 \\
 42 & 154 & 211 & 129 & 31 & 1 & 0 \\
 132 & 582 & 1014 & 871 & 369 & 63 & 1 \\
\end{array}
\right).$$ The generating function of this triangle is $c(x(r+1)(1-rx))$. The initial column consists of the Catalan numbers. The row sums are given by $r=1$, and thus they have generating function $c(2x(1-x))$. The row sums \seqnum{A118376} begin
$$1, 2, 6, 24, 112, 568, 3032, 16768,\ldots.$$ They count the number of all trees of weight $n$, where nodes have positive integer weights and the sum of the weights of the children of a node is equal to the weight of the node.

For $a=0$, $b=1$, we get the triangle that begins 
$$\left(
\begin{array}{ccccccc}
 1 & 0 & 0 & 0 & 0 & 0 & 0 \\
 1 & 1 & 0 & 0 & 0 & 0 & 0 \\
 1 & 3 & 1 & 0 & 0 & 0 & 0 \\
 1 & 7 & 7 & 1 & 0 & 0 & 0 \\
 1 & 15 & 30 & 15 & 1 & 0 & 0 \\
 1 & 31 & 103 & 103 & 31 & 1 & 0 \\
 1 & 63 & 312 & 505 & 312 & 63 & 1 \\
\end{array}
\right).$$ 
For $a=1$, $b=1$, we get the triangle that begins 
$$\left(
\begin{array}{ccccccc}
 1 & 0 & 0 & 0 & 0 & 0 & 0 \\
 1 & 1 & 0 & 0 & 0 & 0 & 0 \\
 2 & 4 & 1 & 0 & 0 & 0 & 0 \\
 5 & 15 & 11 & 1 & 0 & 0 & 0 \\
 14 & 56 & 69 & 26 & 1 & 0 & 0 \\
 42 & 210 & 364 & 252 & 57 & 1 & 0 \\
 132 & 792 & 1770 & 1800 & 804 & 120 & 1 \\
\end{array}
\right).$$ This triangle has generating function $c(x(r+1-r^2x))$. When $r=1$, we get the generating function of the row sums. These begin 
$$1,2,7,32,166, \ldots,$$ and they count the number of ordered rooted trees with $n$ generators (\seqnum{A108524}). Their generating function is thus $c(x(2-x))$. 

For $a=1$, $b=2$, we get the triangle that begins 
$$\left(
\begin{array}{ccccccc}
 1 & 0 & 0 & 0 & 0 & 0 & 0 \\
 1 & 1 & 0 & 0 & 0 & 0 & 0 \\
 1 & 4 & 1 & 0 & 0 & 0 & 0 \\
 1 & 11 & 11 & 1 & 0 & 0 & 0 \\
 1 & 26 & 58 & 26 & 1 & 0 & 0 \\
 1 & 57 & 226 & 226 & 57 & 1 & 0 \\
 1 & 120 & 747 & 1296 & 747 & 120 & 1 \\
\end{array}
\right).$$
Corresponding to the Eulerian triangle $E_B$ of type $B$, we have for $a=3$, $b=4$ the triangle that begins 
$$\left(
\begin{array}{ccccccc}
 1 & 0 & 0 & 0 & 0 & 0 & 0 \\
 1 & 1 & 0 & 0 & 0 & 0 & 0 \\
 1 & 6 & 1 & 0 & 0 & 0 & 0 \\
 1 & 19 & 19 & 1 & 0 & 0 & 0 \\
 1 & 48 & 126 & 48 & 1 & 0 & 0 \\
 1 & 109 & 562 & 562 & 109 & 1 & 0 \\
 1 & 234 & 2031 & 3916 & 2031 & 234 & 1 \\
\end{array}
\right).$$
The generating function of this triangle is $c(x(r+1-x(r-1)^2))$. 
When $r=1$, we get $c(2x)$. Thus the row sums of this triangle are $2^n C_n$ \seqnum{A151374}.

We have the following result, which follows from standard techniques of Riordan arrays applied to orthogonal polynomials. 
\begin{proposition} The generating function $\mathcal{T}(G(x,r;a,b))$ is the generating function of the coefficient array for the moment polynomials of the family of orthogonal polynomials whose coefficient array is given by the ordinary Riordan array
$$\left(\frac{1+(r+1)x}{1+2(r+1)x+((r+1)(a+1)-b)x^2}, \frac{x}{1+2(r+1)x+((r+1)(a+1)-b)x^2}\right).$$ 
The corresponding moment array is given by 
$$\left(\mathcal{T}(G(x,r;a,b)), \frac{x}{1-2(r+1)x}c\left(\frac{x^2((r+1)(a+1)-b)}{(1-2(r+1)x)^2}\right)\right).$$
\end{proposition}

\section{Acknowledgements} This note makes reference to many sequences to be found in the OEIS, which at the time of writing contains more than $300,000$ sequences. All who work in the area of integer sequences are profoundly indebted to Neil Sloane.

\bigskip
\hrule
\bigskip
\noindent 2010 {\it Mathematics Subject Classification}: Primary
11B83; Secondary 33C45, 42C05, 15B36, 11C20, 05A15, 05E45, 	44A10
\noindent \emph{Keywords:} Sumudu transform, series reversion, Eulerian triangle, Riordan array, orthogonal polynomial, moment sequence.

\bigskip
\hrule
\bigskip
\noindent (Concerned with sequences
\seqnum{A000108},
\seqnum{A000165},
\seqnum{A000670},
\seqnum{A008292},
\seqnum{A060187},
\seqnum{A108524},
\seqnum{A114608},
\seqnum{A118376},
\seqnum{A123125},
\seqnum{A151374}, and
\seqnum{A173018}.)

\end{document}